\documentclass{amsart}
\usepackage[english]{babel}
\usepackage{amsmath,amsfonts,amssymb,amsthm}
\usepackage[all]{xy}
\usepackage{tikz}
\usepackage[mathcal]{eucal}
\usepackage{mathrsfs}
\usepackage{caption}
\usepackage{fmtcount}
\usepackage{hyperref}
\newtheorem{lm}{Lemma}[section]
\newtheorem{prop}[lm]{Proposition}
\newtheorem{thm}[lm]{Theorem}
\newtheorem{cor}[lm]{Corollary}
\newtheorem{df}[lm]{Definition}
\theoremstyle{definition}
\newtheorem{rk}[lm]{Remark}
\newtheorem{ex}[lm]{Example}
\theoremstyle{remark}

\newcommand{\Spec}{\mathrm{Spec}\,}
\newcommand{\C}{\mathbb{C}}
\newcommand{\R}{\mathbb{R}}

\newcommand{\regsh}{\mathcal{O}}

\newcommand{\Q}{\mathbb{Q}}

\renewcommand{\div}{\mathrm{div}}

\newcommand{\e}{\varepsilon}
\newcommand{\Pspace}{\boldsymbol{\mathrm{P}}}

\newcommand{\Rscr}{\mathscr{R}}
\newcommand{\Proj}{\mathrm{Proj}}

\newcommand{\Fscr}{\mathscr{F}}

\newcommand{\rank}{\mathrm{rk}\,}

\title{Ample Weil divisors}
\author{alberto chiecchio and stefano urbinati}
\date{}
\begin{document}
\maketitle
\tableofcontents
\begin{abstract}We define and study positivity (nefness, amplitude, bigness and pseudo-effectiveness) for Weil divisors on normal projective varieties. We prove various characterizations, vanishing and non-vanishing theorems for cohomology, global generations statements, and a result related to log Fano.
\end{abstract}
\section{Introduction}
Positivity is a major topic in modern algebraic geometry, \cite{lazarsfeld}, and is originally defined for invertible sheaves on projective varieties. In recent years, there have been efforts to extend this notion to a broader context. Using Shokurov's $b$-divisors, in \cite{BFJ} and \cite{BdFF} the authors defined \emph{nef Weil $b$-divisors} and the \emph{nef envelope}. Independently, in \cite{stefano}, motivated by the pullback for Weil divisors defined in \cite{dFH}, the second author introduced \emph{nef Weil divisors}. These two definitions arise naturally and have good properties in many different ways.

Our motivation for studying positivity for reflexive sheaves comes from the Minimal Model Program, and the properties described in \cite{stefano} seem to be more connected with our scope. In this paper, starting from the second approach, we define (relative) nefness, amplitude, bigness and pseudo-effectiveness for Weil divisors, and we prove most of the properties one would expect from such definitions (vanishing, non-vanishing, and global generation theorems).

In section \ref{ndef} we first introduce and compare the two notions of nefness for Weil divisors (coming from the two possible approaches of \cite{BdFF} and \cite{stefano}). We then give our definition of ample, big and pseudo-effective reflexive sheaf and describe the first natural properties.

In section \ref{char} we prove some characterizations, that will emphasize why our definitions are natural.

In section \ref{van} we focus on vanishing theorems. We prove Serre's and Fujita's vanishing for locally free sheaves, and Kawamata-Viehweg vanishing.

In the last section we look more directly toward the Minimal Model Program. In section \ref{nvan} we prove non-vanishing and global generation. We conclude, in section \ref{fano}, by showing that for any lt${}^+$ variety with anticanonical ample divisor, there exists a boundary $\Delta$ that makes it log Fano in the usual sense.

The definitions and most of the results hold in any characteristic (as well as in mixed characteristic). We will restrict to characteristic $0$ only for Kawamata-Viehweg and for the last section.
\subsection*{Acknowledgments} We would like to thank T. de Fernex for the interesting discussions, the constant disponibility and for suggesting the definition of $b$-nef divisors. We would like to thank C. Casagrande, C. D. Hacon, S. Kov\'acs, A. Langer, M. Musta\c t\u a, Zs. Patakfalvi, K. Schwede for the interest showed in the project and helpful suggestions. We would like to thank C. Aponte Rom\'an for compelling us to work in mixed characteristic as long as possible. We are grateful to the referees for the several useful suggestions.

\section{Positive Weil divisors}
\label{ndef}
\textbf{Notation} Throughout this paper, all schemes will be of finite type over a field, which will alway be denoted by $k$, of arbitrary characteristic. Almost all the definitions are true in bigger generality (for example when $k$ is a Noetherian ring). We will write $1||m$ for $m$ sufficiently divisible.
\subsection{Two possible definitions of nef Weil divisor}
In this first part we will discuss two different approaches to the study of Weil divisors. The first one focuses more on the structure of the associated reflexive sheaf, the second one relies on the concept of $b$-divisors and was suggested to the authors by Professor T. de Fernex.
\subsubsection{\textbf{First approach}} Let $X$ be a normal variety over a noetherian ring $k$. Recall that, if $D$ is a Weil $\Q$-divisor,
$$
\Rscr(X,D):=\bigoplus_{m\geq0}\regsh_X(\lfloor mD\rfloor)
$$
is an $\regsh_X$-algebra. If $f:X\rightarrow U$ is a proper morphism of normal varieties and $D$ is a Weil $\Q$-divisor, we have
$$
\Rscr(X/U,D):=\bigoplus_{m\geq0}f_*\regsh_X(\lfloor mD\rfloor),
$$
which is naturally an $\regsh_U$-algebra.
\begin{rk}
If $f$ has connected fibers, $f_*\regsh_X=\regsh_U$. This is the case, for example, when $X\rightarrow U$ is proper birational, or if $X$ is connected and $U=\Spec k$. In general, the structure of $\regsh_U$-algebra on $\Rscr(X/U,D)$ is given via the natural map $\regsh_U\rightarrow f_*\regsh_X$.
\end{rk}
If $\Fscr$ is a coherent sheaf on $X$, we say that it is \emph{relatively globally generated} if $f^*f_*\Fscr\rightarrow\Fscr$ is surjective.
\begin{rk} This condition is local on the base, that is, it can be checked on open affine subschemes of $U$, where it reduces to the usual global generation.
\end{rk}

From now on we will restrict to the case of a projective morphism of quasi-projective normal varieties $X\rightarrow U$.
\begin{df} Let $f:X\rightarrow U$ be a projective morphism of quasi-projective normal varieties. A Weil $\Q$-divisor $D$ on $X$ is \emph{relatively asymptotically globally generated}, in short \emph{relatively agg} or \emph{$f$-agg}, if, for every positive integer $1||m$, $\regsh_X(mD)$ is relatively globally generated.
\end{df}
\begin{df}[cfr. \cite{stefano2}, 4.1]  Let $f:X\rightarrow U$ be a projective morphism of quasi-projective normal varieties over $k$. A Weil $\Q$-divisor $D$ on $X$ is \emph{relatively nef}, or \emph{$f$-nef}, if for every relatively ample $\Q$-Cartier divisor $A$, $\regsh_X(D+A)$ is $f$-agg.

If $U=\Spec k$ (so that $X$ is projective) we will simply say that $D$ is \emph{nef}.
\end{df}

\begin{rk}\label{rk:failadditivity} %This definition of nefness is expected to be \emph{not} additive. Indeed, let $X$ be a projective variety (over a Noetherian ring $k$) and $D_1$ and $D_2$ two nef Weil divisors such that for all $m\geq1$, $\regsh_X(mD_1)\otimes\regsh_X(mD_2)\neq\regsh_X(mD_1+mD_2)$. Then, for any ample Cartier divisor $A$ and for any $m\gg0$, $\regsh_X(mA+mD_1)$ and $\regsh_X(mA+mD_2)$ are globally generated; thus so is $\regsh_X(mA+mD_1)\otimes\regsh_X(mA+mD_2)$. However, $\regsh_X(mA+mD_1)\otimes\regsh_X(mA+mD_2)\neq\regsh_X(m2A+m(D_1+D_2))$, and the sheaf on the right is the reflexive hull of the sheaf on the left. Hence, the two sheaves have the same global sections: therefore $\regsh_X(m2A+m(D_1+D_2))$ cannot be globally generated, which implies that $D_1+D_2$ is \emph{not} nef. 
It is not clear wether this notion is additive. Let $D_1$ and $D_2$ be  two nef Weil divisors such that for all $m\geq1$, $\regsh_X(mD_1)\otimes\regsh_X(mD_2)\neq\regsh_X(mD_1+mD_2)$. Then, for any ample Cartier divisor $A$ and for any $m\gg0$, $\regsh_X(mA+mD_1)$ and $\regsh_X(mA+mD_2)$ are globally generated; thus so is $\regsh_X(mA+mD_1)\otimes\regsh_X(mA+mD_2)$. However, $\regsh_X(mA+mD_1)\otimes\regsh_X(mA+mD_2)\neq\regsh_X(m2A+m(D_1+D_2))$, and the sheaf on the right is the reflexive hull of the sheaf on the left. Thus it is not immediately possible to deduce the global generation of $\regsh_X(m2A+m(D_1+D_2))$.
\end{rk}
\subsubsection{\textbf{Second approach}}
The following notation is from \cite{BdFF}. In this subsection and the next one, the field $k$ is assumed to be algebraically closed of characteristic $0$.

The relative Riemann-Zariski space of $X$ is the projective limit
$$
\mathfrak{X}:=\varprojlim_{\pi} X_{\pi}
$$
where $\pi: X_{\pi} \to X$ is a proper birational morphism.

The group of Weil and Cartier $b$-divisors over $X$ are respectively defined as
\begin{eqnarray*}
\mathrm{Div}(\mathfrak{X}):= \varprojlim_{\pi} \mathrm{Div}(X_{\pi}) \quad \textrm{and}\\
\mathrm{CDiv}(\mathfrak{X}):= \varinjlim_{\pi} \mathrm{CDiv}(X_{\pi}),
\end{eqnarray*}
where the first limit is taken with respect to the pushforward and the second with respect to the pullback. We will say that a Cartier $b$-divisor $C$ is \emph{determined} on $X_{\pi}$ for a model $X_{\pi}$ over $X$ such that $C_{\pi'}=f^*C_{\pi}$ for every $f: X_{\pi'} \to X_{\pi}$.

Let $\mathfrak{a}$ be a coherent fractional ideal sheaf on $X$. If $X_{\pi} \to X$ is the normalized blow-up of $X$ along $\mathfrak{a}$, we will denote $Z(\mathfrak{a})$ the Cartier $b$-divisor determined by $\mathfrak{a}\cdot \regsh_{X_{\pi}}$. 

\begin{lm}[\cite{BdFF}, 1.8] A Cartier $b$-divisor $C \in \mathrm{CDiv}(\mathfrak{X})$ is of the form $Z(\mathfrak{a})$ if and only if $C$ is relatively globally generated over $X$.
\end{lm}

\begin{df}[\cite{BdFF}, 2.3] Let $D$ be an $\R$-Weil divisor on $X$. The \emph{nef envelope} $\mathrm{Env}(D)$ is the $\R$-Weil $b$-divisor $Z(\mathfrak{a}_{\bullet})$ associated to the graded sequence $\mathfrak{a}_{\bullet}=\{\regsh_X(mD)\}_{m \geq 1}$, where $$Z(\mathfrak{a}_{\bullet}) := \lim_{m\to \infty} \frac{1}{m} Z(\mathfrak{a}_m).$$
\end{df}

\begin{lm}[\cite{BdFF}, 2.10 and 2.12] We have that 
\begin{enumerate}
\item  $Z(\mathfrak{a}_{\bullet})$ is $X$-nef and

\item $\mathrm{Env}(D)$ is the largest $X$-nef $\R$-Weil $b$-divisor $W$ such that $\pi_*(W_{\pi})\leq D$ for all $\pi$.
\end{enumerate}
\end{lm}

\begin{df} We will say that an $\mathbb{R}$-divisor $D$ on $X$ is \emph{$b$-nef} if $\mathrm{Env}(D)$ is a nef $b$-divisor, in the sense that it is a limit of Cartier $b$-divisors determined by nef divisors. 
\end{df}

\begin{prop}\label{prop:b-nef is additive} Nef $b$-divisors are additive, i.e. if $N_1, N_2$ are nef, then $N_1+N_2$ is nef as well.
\end{prop}
\begin{proof}
This comes directly from the fact the sum of two limits is the limit of the sum.
\end{proof}
\subsubsection{\textbf{Relations}} This section is devoted to the comparison of the two definitions. Let us recall that in this subsection we are working over an algebraically closed field of characteristic $0$. 
\begin{prop} Let $X$ be a normal variety and $D$ be a Weil $\R$-divisor. If $\Rscr(X, D)$ is a finitely generated $\regsh_X$-algebra, then $D$ is nef if and only if $D$ is $b$-nef. 
\end{prop}
\begin{proof} Since $\mathrm{Env}(D)$ is determined at the level of the small map from $\Proj_X\Rscr(X,D)$, it is an easy consequence of \cite[4.3]{stefano2}.
\end{proof}

\begin{prop}
Let $D$ be an $\mathbb{R}$-divisor on $X$. If $D$ is nef, then $D$ is $b$-nef.
\end{prop}
\begin{proof} By definition of nefness, for every $\Q$-Cartier ample divisor $A$, $\regsh_X(m(D+A))$ is globally generated for all positive $m$ divisible by a given $m_0$. Let us fix an $m$ and let us consider a log resolution $f:Y\rightarrow X$ of $\regsh_X(mD)$. Then $\regsh_Y\cdot \regsh_X(m(D+A))$ is globally generated. Since we have
$$\regsh_Y\cdot\regsh_X(m(D+A)) = [\regsh_Y\cdot\regsh_X(m(D+A))]^{\vee \vee},$$
$\div(\regsh_Y\cdot\regsh_X(m(D+A))$ is nef, thus so is the Cartier $b$-divisor determined by it. Since this is true for every $m$ divisible by $m_0$, ${\rm Env}(D+A)=\frac{1}{m_0}{\rm Env}(m_0(D+A))$, by \cite[2.6]{BdFF}, ${\rm Env}(D+A)$ is the limit of Cartier $b$-divisors determined by nef divisors.

Since this is true for any $A$, and $\lim_n {\rm Env}(B_n)={\rm Env}(\lim_n B_n)$, \cite[2.7]{BdFF}, ${\rm Env}(D)$ is the limit of Cartier $b$-divisors determined by nef divisors as well, which implies that $D$ is $b$-nef.
\end{proof}

\begin{rk} We do not know if the converse implication holds in general. Since $b$-nefness is additive, as consequence of \ref{prop:b-nef is additive}, a failure of additivity for nefness, see remark \ref{rk:failadditivity}, would imply that the two notions are not equivalent.
\end{rk}

\subsection{Definition and first properties of ample divisors}

\begin{df} Let $f:X\rightarrow U$ be a projective morphism of quasi-projective normal varieties over $k$. A Weil $\Q$-divisor $D$ on $X$ is called \emph{relatively almost ample}, or \emph{almost $f$-ample}, if, for every relatively ample $\Q$-Cartier divisor $A$ there exists a $b>0$ such that $bD-A$ is $f$-nef. It is called \emph{relatively ample}, or \emph{$f$-ample}, if, in addition, $\Rscr(X/U,D)$ is finitely generated.

If $U=\Spec k$, in the two cases above we will simply say \emph{almost ample} or \emph{ample}.
\end{df}
\begin{rk} If $X$ is projective over $\C$, and $(X,\Delta)$ is a klt pair, then for all $\Q$-divisors, the algebras $\Rscr(X/U,D)$ are finitely generated $\regsh_X$-algebras, \cite[92]{kol}, thus the notions of almost ample and ample coincide. On the other hand, in general almost ample is not equivalent to ample, as example \ref{ex:almost ample not ample} shows.
\end{rk}
\begin{df} Let $f:X\rightarrow U$ be a projective morphism of quasi-projective normal varieties over $k$. A Weil $\Q$-divisor $D$ is called \emph{relatively big}, or \emph{f-big}, if there exist an $f$-ample $\Q$-Cartier divisor $A$ and an effective Weil $\Q$-divisor $E$ such that $D\sim_{f,\Q} A+E$.

If $U=\Spec k$, we will simply say that $D$ is \emph{big}.
\end{df}
\begin{df} Let $f:X\rightarrow U$ be a projective morphism of quasi-projective normal varieties over $k$. A Weil $\Q$-divisor $D$ is called \emph{relatively pseudo-effective}, or \emph{$f$-pseff}, if for every $f$-ample $\Q$-Cartier divisor $A$, $D+A$ is $f$-big.

If $U=\Spec k$, we will simply say that $D$ is \emph{pseudo-effective}.
\end{df} 
\begin{lm}\label{lm:immediate properties} Let $f:X\rightarrow U$ be a projective morphism of normal quasi-projective varieties over $k$. Let $D$ be a Weil $\Q$-divisor on $X$. We have the following properties.
\begin{enumerate}
\item Assume that $D$ is a $\Q$-Cartier divisor and $k$ is an algebraically closed field. Then the above notions of (relative) nefness, bigness, pseudo-effectiveness and amplitude coincide with the usual ones.
\item Relative nefness, bigness, pseudo-effectiveness and (almost) amplitude are asymptotic conditions, that is, $D$ is relatively nef / big / pseff / (almost) ample if and only if for any, or all, $m>0$, $mD$ is relatively nef / big / pseff / (almost) ample.
\item Let $D$ be an $f$-nef divisor, and let $L$ be a relatively agg $\Q$-Cartier divisor. Then $D+L$ is relatively nef.
\item Let $D$ be a Weil $\Q$-divisor. Then $D$ is relatively almost ample if and only if, for any $\Q$-Cartier divisor $L$, there exists a $b>0$ such that $bD-L$ is relatively nef.
\item Let $D$ be relatively almost ample. Then $D$ is relatively nef.
\item Let $D$ be relatively almost ample. Then $D$ is relatively agg.
\item\label{it:ample=>big} Let $D$ be relatively almost ample and $k$ be an algebraically closed field. Then $D$ is $f$-big.
\item Let $D$ be relatively big. Then it is relatively pseudo-effective.
\item\label{it:D+A 1} Let $D$ be a relatively agg Weil $\Q$-divisor, and let $A$ be a relatively ample $\Q$-Cartier divisor. Then $D+A$ is relatively almost ample. If moreover $\Rscr(X,D)$ is finitely generated, $D+A$ is relatively ample.
\item Let $D$ be a relatively (almost) ample Weil $\Q$-divisor and let $A$ be a relatively agg $\Q$-Cartier divisor. Then $D+A$ is relatively (almost) ample.
\item Let $D$ be a Weil $\Q$-divisor. Then $D$ is relatively (almost) ample if and only if, for any $\Q$-Cartier divisor $L$, there exists a $b>0$ such that $bD-L$ is relatively (almost) ample.
\end{enumerate}
\end{lm}
\begin{proof}\begin{enumerate}
\item Well known.
\item Clear. Notice that, if $D$ is a Weil divisor, $\Rscr(X,D)$ is finitely generated if and only if, for some $a>0$, $\Rscr(X,aD)$ is finitely generated, \cite[5.68]{hk}.
\item Let $A$ be an ample Cartier divisor. By hypothesis $\regsh_X(m(D+A))$ is relatively globally generated for all $1||m$. Since $L$ is relatively agg, $\regsh_X(m(D+A))\otimes\regsh_X(mL)\cong\regsh_X(m(D+L+A))$ is relatively agg.
\item The if is immediate. Let us show the only if. Let $L$ be any $\Q$-Cartier divisor. Without loss of generality, we can assume that is Cartier. Let $A$ be a relatively ample $\Q$-Cartier divisor such that $A+L$ is relatively ample. By definition there exists a $b>0$ such that $bD-A-L$ is relatively nef. Since $A$ is relatively agg, $bD-L=(bD-A-L)+A$ is relatively nef.
\item In the previous part, choose $L=0$.
\item Let $A$ be relatively ample $\Q$-Cartier divisor, and let $b>0$ such that $bD-A$ is relatively nef. By definition, then, for every relatively ample $\Q$-Cartier divisor $L$, $bD-A+L$ is relatively agg. Setting $A=L$, we get that $bD$ is relatively agg, that is, $D$ is relatively agg.
\item Let $A$ be a $\Q$-Cartier $f$-ample divisor, and let $b>0$ such that $bD-A$ is $f$-agg. Let $m>0$ such that $m(bD-A)$ is relatively globally generated. Let $V\subseteq U$ be an open affine subscheme, such that the complement of $f^{-1}V$ has codimension at least $2$ in $X$. Then, $m(bD-A)$ has a section on $f^{-1}(V)$, which extends to a global section (the sheaf $\regsh_X(mbD-mA)$ is $S_2$). Let $M\in|\regsh_X(mbD-mA)|$, $M\geq0$, constructed as explained before; then $D\sim_{f,\Q}\frac{1}{mb}A+\frac{1}{m}M$.
\item Clear, since if you add a relatively ample Cartier divisor to a relatively big divisor, you still have a relatively big Weil divisor.
\item Since all the conditions are asymptotic, it is enough to work with Cartier divisors. In particular, we can assume that $A$ is Cartier. Let $L_1$ be a relatively ample Cartier divisor, and let $b>0$ such that $bA-L_1$ is relatively very ample. Then, for every relatively ample $\Q$-Cartier $\Q$-divisor $L_2$, $bA-L_1+L_2$ is relatively agg. Thus $bD+(bA-L_1+L_2)=b(D+A)-L_1+L_2$ is relatively agg. By definition, this means that $bD+bA-L_1$ is relatively nef. Therefore, $D+A$ is relatively almost ample. For the second part of the statement, notice that $\regsh_X(m(D+A))=\regsh_X(mD)\otimes\regsh_X(mA)$ since $A$ is Cartier.
\item Let $L$ be a relatively ample $\Q$-Cartier divisor and let $b>0$ such that $bD-L$ is relatively nef. Thus $bD-L+bA=b(D+A)-L$ is relatively nef, which implies that $D+A$ is relatively almost ample. As in the previous part, the finite generation of the sheaf algebra is preserved by adding a Cartier divisor, and this concludes the proof.
\item The if is clear. Let $D$ be (almost) relatively ample and let $L$ be $\Q$-Cartier relatively ample. Then there exists $b>0$ such that $bD-3L$ is relatively nef. Hence $bD-2L=bD-3L+L$ is relatively agg (by definition of relative nefness), and thus $bD-L=bD-2L+L$ is relatively (almost) ample.
\end{enumerate}
\end{proof}
\begin{ex} Let us fix an algebraically closed field $k$. Let $X$ be the quadric cone $X=\{xy-wz=0\}$ in $\Pspace^4$, which is a cone over an embedding of type $(1,1)$ of $Q\cong\Pspace^1\times\Pspace^1$. Let $C_D$ be a divisor on $X$, which is a cone over a divisors $D$ in $Q$, of type $(a,b)$. Using the above lemma it is not hard to see that $C_D$ is almost ample if and only if $a,b>0$, that is, if and only if $D$ is ample, while $C_D$ is nef if and only if $a,b\geq0$, that is, if and only if $D$ is nef. If $k=\C$, then $X$ is log terminal, and thus all algebras of local sections are finitely generated, and hence we can say that $C_D$ is ample if and only if $D$ is ample.
\end{ex}
\begin{ex}\label{ex:almost ample not ample} We will give an example of a divisor which is almost ample, but not ample. Let $X$ be a normal variety with isolated singularities and let $D$ be a Weil divisor such that $\Rscr(X,D)$ is not finitely generated. For example, $X$ can be the cone of \cite[\S3]{stefano}, and $D=K_X$. Let $A$ be any Cartier ample divisor on $X$. By the above lemma and \cite[5.5]{stefano}, for $b$ sufficiently large $bA+D$ is almost ample. On the other hand, since $A$ is Cartier and $\Rscr(X,D)$ is not finitely generated, $\Rscr(X,bA+D)$ is not finitely generated, showing that $bA+D$ is not ample.
\end{ex}
%
%
%
%
%          CHARACTERIZATIONS
%
%
%
%
\section{Characterizations}
\label{char}
%
%
%
%
%          NEF WEIL DIVISORS
%
%
%
%
The characterizations of this section are of two types. Some of them are the equivalent of well-known results for $\Q$-Cartier divisors, while other investigate the behavior of positivity of Weil divisors under $\Q$-Cartierizations. For this purpose, before starting our discussion, let us recall a well known result, \cite[6.2]{komo}, which we will use extensively.
\begin{lm}\label{lm:themostusefullemmaintheworld} Let $X$ be a normal algebraic variety and let $B$ be a Weil divisor on it. The following are equivalent:
\begin{enumerate}
\item $\Rscr(X,B)$ is a finitely generated $\regsh_X$-algebra;
\item there exists a small, projective birational morphism $f:Y\rightarrow X$ such that $Y$ is normal, and $\bar{B}:=f^{-1}_*B$ is $\Q$-Cartier and $f$-ample.
\end{enumerate}
Moreover, the morphism $f:Y\rightarrow X$ in (b), if it exists, is unique up to isomorphism, namely $Y\cong\Proj_X\Rscr(X,D)$; also, for all $m\geq0$, $f_*\regsh_Y(m\bar{B})=\regsh_X(mB)$.
We will call such a map the \emph{$\Q$-Cartierization} of $B$.
\end{lm}
\subsection{Nef Weil divisors}
\emph{}

The following is a generalization of \cite[1.4.10]{lazarsfeld} to our setting.

\begin{prop} Let $f:X\rightarrow U$ be a projective morphism of normal quasi-projective varieties over $k$. Let $D$ be a divisor on $X$ and let $H$ be an $f$-ample $\Q$-Cartier $\Q$-divisor. Then $D$ is $f$-nef if and only if, for all sufficiently small $\e\in\Q$, $0<\e\ll1$, $D+\e H$ is $f$-ample.
\end{prop}
\begin{proof} Let us assume that $D+\e H$ is $f$-ample for all $\e\in\Q$, $0<\e\ll1$. Let $A$ be an $f$-ample $\Q$-Cartier $\Q$-divisor. For rational positive $\e$ sufficiently small, $A-\e H$ is $f$-ample. Since $H$ is $\Q$-Cartier, for all positive $1||m$,
$$
\regsh_X(m(D+A))=\regsh_X(m((D+\e H)+(A-\e H)))\cong\regsh_X(m(D+\e H))\otimes\regsh_X(m(A-\e H)).
$$
Since both $D+\e H$ and $A-\e H$ are $f$-ample, they are $f$-agg, and thus so is $D+A$. By definition, $D$ is $f$-nef.

Conversely, let $D$ be $f$-nef. By substituting $H$ with $\e H$, we reduce to prove that $D+H$ is $f$-ample. %Let us assume that $H$ is $\Q$-Cartier. Then $D+H/2$ must be $f$-agg, and thus $D=D+H/2+H/2$ is $f$-ample. Let us assume that $D$ is Cartier. 
Let $A$ be $\Q$-Cartier and $f$-ample. For $1\gg\delta>0$, $H-\delta A$ is $f$-ample; moreover $D+\delta A$ is $f$-agg. By the hypothesis, $H-\delta A$ is $\Q$-Cartier; hence $D+H=(D+\delta A)+(H-\delta A)$ is $f$-ample.
\end{proof}
\begin{thm}\label{thm:nef<=>qnef} Let $X\rightarrow U$ be a projective morphism of normal projective varieties over an algebraically closed field $k$. If $g:Y\rightarrow X$ is a small projective birational map over $U$ such that $\bar{D}:=g_*^{-1}D$ is $\Q$-Cartier and $g$-ample, then $D$ is nef over $U$ if and only if $\bar{D}$ is nef over $U$.
\end{thm}
\begin{rk} This is the relative and characteristic free version of \cite[4.3]{stefano2}. The strategy of our proof is to reduce to an affine base and then to use the \cite[4.3]{stefano2}. We point out that, although the referred result is only stated in characteristic zero, its proof works over any algebraically closed field.
\end{rk}
\begin{proof} Let $D$ be nef over $U$. By definition, for every Cartier divisor $A$ ample over $U$, $\regsh_X(m(D+A))$ is relatively globally generated (over $U$) for $1||m$. Since $\bar{D}$ is $g$-ample, we also have that $\regsh_Y(m\bar{D})$ is globally generated over $X$ for $1||m$. We claim that this implies that $\regsh_Y(m(\bar{D}+g^*A))$ is globally generated over $U$. Since both the assumption and the claim (and the properties of the map $g$ and the divisor $\bar{D}$) are local on the base $U$, we can assume $U$ to be affine. Thus we can assume that $\regsh_X(m(D+A))$ is globally generated and we will show that so is $\regsh_Y(m(\bar{D}+g^*A))$.  This is proven in \cite[4.3]{stefano2}. Going back to the relative setting, we showed that $\regsh_Y(m(\bar{D}+g^*A))$ is relatively globally generated over $U$; thus it is relatively nef. Since nefness for $\Q$-Cartier divisors is a closed property, \cite[1.4.23]{lazarsfeld}, $\bar{D}$ is relatively nef.

Now let us suppose that $\bar{D}$ is relatively nef over $U$, and let $A$ be a Cartier divisor on $X$, ample over $U$. The claim is that, for $1||m$, $\regsh_X(m(D+A))$ is globally generated over $U$. As in the previous part, we can assume that $U$ is affine. Because of \cite[2.15]{keeler}, nefness is an absolute notion, so $\bar{D}$ is nef. Then, the same proof of \cite[4.3]{stefano2} shows that $\regsh_X(m(D+A))$ is globally generated, thus concluding our proof.
\end{proof}
%
%
%
%
%          AMPLE WEIL DIVISORS
%
%
%
%
\subsection{Ample Weil divisors}
\begin{prop}[see \cite{har}, II.7] Let $X\rightarrow U$ be a projective morphism of normal quasi-projective varieties and let $D$ be a relatively almost ample Weil $\Q$-divisor. For any coherent sheaf $\Fscr$ and for all positive $1||m$, $\regsh_X(mD)\otimes\Fscr$ is relatively globally generated.

Conversely, let $D$ be a Weil $\Q$-divisor such that $\Rscr(X,D)$ is finitely generated. If for any coherent sheaf $\Fscr$ and for all positive $1||m$ the sheaf $\Fscr\otimes\regsh_X(mD)$ is relatively globally generated, then $D$ is relatively ample.
\end{prop}
\begin{proof} Notice that, for any relatively ample $\Q$-Cartier divisor $A$, there exists $b>0$ such that $bD-A$ is relatively agg. Indeed, there exists $b>0$ such that $bD-2A$ is relatively nef (by definition). But then $bD-A=(bD-2A)+A$ is relatively agg. Since $A$ is relatively ample, $\Fscr\otimes\regsh_X(mA)$ is relatively globally generated for positive $1||m$. Thus
\begin{eqnarray*}
\Fscr\otimes\regsh_X(mbD)\cong\Fscr\otimes\regsh_X(mA)\otimes\regsh_X(mbD-mA)
\end{eqnarray*}
is relatively globally generated for positive $1||m$.

Now let us prove the converse. Since we are working with asymptotic conditions, we can substitute $D$ with $aD$ for a positive $1||a$, so that $\Rscr(X,D)$ is generated in degree $1$, that is, $\regsh_X(mD)\cong\regsh_X(D)^{\otimes m}$ for all $m\geq1$. Let $L$ be any relatively ample Cartier divisor, and let $b>0$ such that $\regsh_X(bD-L)$ is relatively globally generated. Since $\regsh_X(bD-L)^{\otimes m}\cong\regsh_X(m(bD-L))$, $bD-L$ is relatively agg. Thus, $bD=(bD-L)+L$ is relatively almost ample, by lemma \ref{lm:immediate properties}(\ref{it:D+A 1}). Since $\Rscr(X,bD-L)$ is finitely generated, because $L$ is Cartier and $\Rscr(X,bD)$ is finitely generated, then $bD$ is relatively ample. Hence, $D$ is relatively ample.
\end{proof}
\begin{thm}\label{thm:ample<=>qample} Let $X\rightarrow U$ be a projective morphism of normal projective varieties over an algebraically closed field $k$. If $g:Y\rightarrow X$ is a small projective birational map over $U$ such that $\bar{D}:=g_*^{-1}D$ is $\Q$-Cartier and $g$-ample, then $D$ is ample over $U$ if and only if $\bar{D}$ is ample over $U$.
\end{thm}
\begin{proof} Let $D$ be ample over $U$. By definition, for every Cartier divisor $H$ ample over $U$, there exists a rational $\e>0$ such that $D-\e H$ is relatively nef. Moreover, $g^{-1}_* (D-\e H)=\bar{D}-\e g^*H$ is $\mathbb{Q}$-Cartier and $g$-ample for $0<\e\ll1$. By theorem \ref{thm:nef<=>qnef} $g^{-1}_* (D-\e H)$ is nef over $U$. Hence, there exists $m\gg0$ such that $g^{-1}_*(mD)= mg^{-1}_* (D-\e H) + g^*(m \e H)$ is ample over $U$.

Similarly, if $\bar{D}=g^{-1}_*D$ is ample over $U$ (and $g$-ample), for every Cartier divisor $H$ on $X$ ample over $U$, there exists a rational $1\gg\e>0$ such that $g^{-1}_*(D) - g^*(\e H)= g^{-1}_*(D - \e H)$ is nef over $U$ and $g$-ample. Then by theorem \ref{thm:nef<=>qnef},  $D - \e H$ is nef over $U$, and this concludes the proof.
\end{proof}
\begin{rk} As observed before, if $D$ is ample, such a map $g:Y\rightarrow X$ must necessarily exist.
\end{rk}
%
%
%
%
%          BIG WEIL DIVISORS
%
%
%
%
\subsection{Big Weil divisors}
\begin{prop}[see \cite{komo}, 3.23]\label{prop:big<=>high rank} Let $f:X\rightarrow U$ be a projective morphism of quasi-projective normal varieties over an algebraically closed field $k$, 
%($U$ irreducible)
and let $D$ be a Weil $\Q$-divisor on $X$. The following are equivalent:
\begin{enumerate}
\item there exist an almost $f$-ample Weil $\Q$-divisor $A$ and an effective divisor $E$ such that $D\sim_{f,\Q} A+E$;
\item there exist an $f$-ample $\Q$-Cartier divisor and an effective divisor $E$ such that $D\sim_{f,\Q} A+E$ ($D$ is $f$-big).
\end{enumerate}
They both imply
\begin{enumerate}
\setcounter{enumi}{2}
\item there exists a $c>0$ such that $\rank f_*\regsh_X(mD)>c m^n$ for all $m>0$, $1||m$, where $n$ is the dimension of the general fiber of $f$.
\end{enumerate}
If $\Rscr(X,D)$ is finitely generated, then they are all equivalent.
\end{prop}
\begin{proof} Clearly (b)$\Rightarrow$(a) and (a)$\Rightarrow$(b) is a consequence of \ref{lm:immediate properties}(\ref{it:ample=>big}). Indeed, you can write $A\sim_{f,\Q} A'+E'$, where $A'$ is $f$-ample and $\Q$-Cartier and $E'$ is effective, and thus $D\sim_{f,\Q} A'+(E'+E)$. It is immediate that (b)$\Rightarrow$(c).

So let us assume (c) and $\Rscr(X,D)$ finitely generated, and prove (b). Since both (b) and (c) are asymptotic conditions, we can assume that $\Rscr(X,D)$ is finitely generated in degree $1$, that is, $\regsh_X(mD)\cong\regsh_X(D)^{\cdot m}$. Moreover, using the same technique as in the proof of \ref{lm:immediate properties}(\ref{it:ample=>big}), we can assume that $U=\Spec R$ is affine. Thus, condition (c) becomes
$$
\rank_R H^0(X,\regsh_X(mD))> c m^n,
$$
for some $c>0$ and all $m>0$, $1||m$.
Let $\mathscr{A}$ be an $f$-very ample line bundle, let $A\in|\mathscr{A}|$ be general enough element, reduced and irreducible (by generality of $\mathscr{A}$ and $A\in|\mathscr{A}|$), and let $i:A\rightarrow X$ be the inclusion.

The first step is to notice that $i^*\regsh_X(D)$ is torsion-free and invertible at the generic point of $A$. Both properties can be checked locally on $X$, so that we can assume that $X$ is affine, since A will intersect the open set where $\regsh_X(D)$ is a line bundle.%. Let $X=\Spec B$, $A=\Spec B/f$, for some irreducible element $f\in B$. By hypothesis, $f$ is not a unit. Let $K=\Frac(B)$ and $L=\Frac(B/f)$. Let $\regsh_X(D)\cong M\modsh$. Moreover, we can always twist $\regsh_X(D)$ with a Cartier divisor and assume that $M\subseteq B$. Under these assumptions, $i^*\regsh_X(D)\cong M\otimes_B B/f$. Let $m\in M\otimes_B B/f$ be torsion, $m\neq0$. Then, without loss of generality, $m$ is the image of an element of $M$, which by abuse of notation we can still denote by $m$. Since the image of $m$ in $M\otimes_B B/f$ is torsion, there exists $b\in B/f$, $b\neq0$, such that $bm=0$ in $M\otimes_B B/f$. Again, by abuse of notation, we can think of $b$ as the class of an element $b\in B$. Thus, we have that $f\nmid m$ ($m\in M\subseteq B$) and $f\nmid b$. However, $f|(bm)$. Since $f$ is irreducible, this is a contradiction. Now, let us compute the rank of $M\otimes_B B/f$. Let us assume, by contradiction, that
%$$
%\rank M\otimes_B B/f=\dim_L M\otimes_B B/f\otimes_{B/f}L\geq2.
%$$
%Let $s_1,s_2\in M\otimes_B B/f$ be two $L$-linearly independent elements. As before, we can assume that they are the image of two elements (again denoted by $s_1$ and $s_2$) in $M$. Since $M$ has rank $1$, there exists a non-trivial combination $\kappa_1 s_1+\kappa_2 s_2=0$, with $\kappa_i\in K$, $i=1,2$, and at least one of the $\kappa_i\neq0$. Since $K$ is the fraction field of $A$, we can assume that $\kappa_i\in B$, $i=1,2$. This relation induces a relation in $M\otimes_B B/f$. Since $s_1$ and $s_2$ are linearly independent in $M\otimes_B B/f$, $f$ must divide both $\kappa_1$ and $\kappa_2$. Substituting $\kappa_i$ with $\kappa_i/f$ ($i=1,2$), we obtain a new relation $\kappa_1 s_1+\kappa_2 s_2=0$, with $k_i\in B$, $i=1,2$. Since we can repeat this process infinitely many times, this implies that each $\kappa_i$ is infinitely divisible by $f$. Since $f$ is not a unit and at least one of the $\kappa_i$ is non-zero, this is a contradiction.

The second step is to restrict to $A$, and use induction. Notice that, by generality, $A$ is normal. We have the short exact sequence
$$
0\rightarrow\regsh_X(-A)\rightarrow\regsh_X\rightarrow\regsh_A\rightarrow 0.
$$
Tensoring with $\regsh_X(mD)$ with get
$$
\regsh_X(mD-A)\rightarrow\regsh_X(mD)\rightarrow\regsh_X(mD)\big|_A\rightarrow 0.
$$
Let $\mathscr{K}_m$ be the kernel of $\regsh_X(mD)\rightarrow\regsh_X(mD)\big|_A$, so that we have the short exact sequence
$$
0\rightarrow\mathscr{K}_m\rightarrow\regsh_X(mD)\rightarrow\regsh_X(mD)\big|_A\rightarrow 0
$$
and $\mathscr{K}_m\subseteq\regsh_X(mD-A)$. The above short exact sequence induces a long exact sequence in cohomology:
$$
0\rightarrow H^0(X,\mathscr{K}_m)\rightarrow H^0(X,\regsh_X(mD))\rightarrow H^0(A,\regsh_X(mD)\big|_A).
$$
As observed before, $\regsh_X(D)\big|_A$ is a torsion-free sheaf of rank $1$. Moreover, $\regsh_X(mD)\big|_A\cong\regsh_X(D)^{\cdot m}\big|_A\cong\big(\regsh_X(D)\big|_A\big)^{\cdot m}$. Since $\regsh_X(D)\big|_A$ is torsion-free of rank $1$, there exists a very ample Cartier divisor $H$ on $A$ over $U=\Spec R$ such that $\regsh_X(D)\big|_A\subseteq\regsh_A(H)$ (as in \cite[2.58]{komo}). Then, for any $m$,
$$
\rank_R H^0(A,\regsh_X(mD)\big|_A)\leq\rank_R H^0(A,\regsh_A(mH))<C m^{n-1},
$$
for some $C>0$. Since, by hypothesis
$$
\rank_RH^0(X,\regsh_X(mD))\geq c m^n,
$$
for some $c>0$, for a positive $m\gg0$, $H^0(X,\mathscr{K}_m)$ has a section. Since $\mathscr{K}_m\subseteq\regsh_X(mD-A)$, $H^0(X,\mathscr{K}_M)\subseteq H^0(X,\regsh_X(mD-A))$. Thus, this section determines an effective divisor $E\sim mD-A$, which concludes the proof.
\end{proof}
We recall, see \cite[2.6 and 2.9]{dFH}, that if $f:X\rightarrow Y$ is a proper birational morphisms of normal varieties and $D$ is a divisor on $X$, the \emph{natural pullback} $f^\natural(D)$ is the divisor associated with $\regsh_Y\cdot\regsh_X(-D)$, that is, $f^\natural(D)$ is the divisor on $Y$ such that
$$
\regsh_Y(-f^\natural(D))=(\regsh_Y\cdot\regsh_X(-D))^{\vee\vee}.
$$
The \emph{pullback} $f^*(D)$ is defined by
$$
f^*(D)=\liminf_m\frac{f^\natural(mD)}{m},
$$
where the infimum is taken point-wise on the coefficients.
\begin{lm}\label{lm:pullback of big} Let $f:Y\rightarrow X$ be a proper birational morphism between normal varieties over $U$ and let $D$ be a Weil divisor on $X$. If $D$ is big over $U$, so are $f^\natural(mD)/m$, $-f^\natural(-mD)/m$, $f^*(D)$ and $-f^*(-D)$, for all $m\geq1$.
\end{lm}
\begin{proof} For simplicity we will only show that $f^\natural(D)$ is big over $U$. The proof proceeds in the same way for all pullbacks. Let $p:X\rightarrow U$ and let $q:Y\rightarrow U$. Since $D$ is $p$-big, we can find a $p$-ample $\Q$-Cartier divisor $A$ and an effective Weil $\Q$-divisor $E$ such that $D\sim_\Q A+E$. Then, $f^\natural(D)=f^\natural(A+E)=f^*(A)+f^\natural(E)$. Since $E$ is effective, so if $f^\natural(E)$. Since $A$ is a $p$-ample $\Q$-Cartier and $f$ is birational, $f^*(A)$ is $q$-big; thus there exist a $q$-ample $\Q$-Cartier divisor $B$ and an effective divisor $E'$ such that $f^*(A)\sim_\Q B+E'$. Hence $f^\natural(D)\sim_\Q B+(E'+f^\natural(E))$, with $B$ $q$-ample and $E'+f^\natural(E)\geq0$.
\end{proof}
\begin{thm}\label{thm:big<=>qbig} Let $X\rightarrow U$ be a projective morphism of normal projective varieties over an algebraically closed field $k$. If $g:Y\rightarrow X$ is a small projective birational map over $U$ such that $\bar{D}:=g_*^{-1}D$ is $\Q$-Cartier and $g$-ample, then $D$ is big over $U$ if and only if $\bar{D}$ is big over $U$.
\end{thm}
\begin{proof} Since $\bar{D}=g^\natural(D)$ ($g$ is small), if $D$ is big over $U$, so is $\bar{D}$ by \ref{lm:pullback of big}.

Let $\bar{D}$ be big over $U$. Let $f:X\rightarrow U$ and let $h=f\circ g:Y\rightarrow U$. For each $m\geq 0$, $g_*\regsh_Y(m\bar{D})=\regsh_X(mD)$. Pushing forward the sheaves, $h_*\regsh_Y(m\bar{D})=f_*g_*\regsh_Y(m\bar{D})= f_*\regsh_X(mD)$. Since $\bar{D}$ is $h$-big over $U$, there exists $c>0$ such that
$$
cm^n<\rank h_*\regsh_Y(m\bar{D})=\rank f_*\regsh_X(mD)
$$
for all $m>0$ sufficiently divisible. By \ref{prop:big<=>high rank}, $D$ is $f$-big.
\end{proof}
%
%
%
%
%          PSEFF WEIL DIVISORS
%
%
%
%
\subsection{Pseudo-effective Weil divisors}
\begin{lm} Let $f:X\rightarrow U$ be a projective morphism of quasi-projective normal varieties over an algebraically closed field $k$ ($U$ irreducible), and let $D$ be a Weil $\Q$-divisor on $X$. Then $D$ is $f$-pseff if and only if, for every $f$-ample $\Q$-Cartier divisor $A$, $A+D\sim_{f,\Q}E$, for some $E$ effective.
\end{lm}
\begin{proof} If $D$ is $f$-pseff, for each $f$-ample $\Q$-Cartier $A$, $D+A$ is $f$-big, and thus clearly $\Q$-linearly equivalent over $U$ to an effective divisor.

Conversely, let $D$ verify the above condition. Let $A$ be any $f$-ample $\Q$-Cartier divisor. Then $D+A/2\sim_{f,\Q}E$, for some $E$ effective. Thus, since $A$ is $\Q$-Cartier, $A+D=A/2+(A/2+D)\sim_{f,\Q}A/2+E$, which is $f$-big.
\end{proof}
\begin{thm} Let $f:X\rightarrow U$ be a projective morphism of normal projective varieties over an algebraically closed field $k$. If $g:Y\rightarrow X$ is a small projective birational map over $U$ such that $\bar{D}:=g^{-1}_*D$ is $\Q$-Cartier and $g$-ample, then $D$ is pseudo-effective over $U$ if and only if $\bar{D}$ is pseudo-effective over $U$.
\end{thm}
\begin{proof} Let $D$ be $f$-pseff. Let $A$ be any $\Q$-Cartier $f$-ample divisor. The map $g$ is small, so $\bar{D}+g^*A=g^{-1}_*(D+A)$; thus the divisor $g^{-1}_*(D+A)$ is still $g$-ample since $g^*A$ is $g$-trivial and $\bar{D}$ is $g$-ample. Moreover, by assumption $D+A$ is $f$-big. Hence, by \ref{thm:big<=>qbig}, $g^{-1}_*(D+A)=\bar{D}+g^*A$ is big over $U$. Since in $\textrm{NS}(X)_\Q$ (i.e., for $\Q$-Cartier divisors) the effective cone is the closure of the big cone, \cite[2.2.26]{lazarsfeld}, $\bar{D}$ is pseudo-effective over $U$.

Conversely, let $\bar{D}$ be pseudo-effective over $U$, and let $A$ be an $f$-ample Cartier divisor on $X$. Since $\bar{D}$ is ample over $X$ and $A$ is ample over $U$, there exists $k\geq2$ such that $\bar{D}+kg^*A$ is ample over $U$. Since $\bar{D}$ is pseudo-effective over $U$ and $\bar{D}+kg^*A$ is ample over $U$, $k(\bar{D}+g^*A)=(k-1)\bar{D}+\bar{D}+kg^*A$ is big over $U$. Thus $\bar{D}+g^*A$ is big over $U$. Since $\bar{D}+g^*A=g^{-1}_*(D+A)$ is $g$-ample, by \ref{thm:big<=>qbig} $D+A$ is big over $U$. Since this is true for every $A$ ample over $U$, $D$ is $f$-pseff.
\end{proof}
%
%
%
%
%          Vanishing theorems
%
%
%
%
\section{Vanishing theorems}
\label{van}
From now on, all our varieties we will be defined over an algebraically closed field $k$.

When talking about positivity of divisors on projective varieties, it is natural to talk about vanishing theorems. Two of the main theorems are Fujita (or Serre) in arbitrary characteristic and Kawamata-Viehweg (or Kodaira) in characteristic zero. It is unlikely that for Weil divisors we have vanishing theorems as strong as for Cartier divisors.
\subsection{Serre and Fujita vanishing}
\begin{thm}[Relative Fujita vanishing for locally free sheaves] Let $f:X\rightarrow U$ be a projective morphism of quasi projective normal varieties, let $H$ be an $f$-ample Weil divisor on $X$ and let $\Fscr$ be a locally free coherent sheaf on $X$. There exists an integer $m(\Fscr,H)$ such that
$$
R^if_*\big(\Fscr\otimes\regsh_X(mH+D)\big)=0,\quad\textrm{for all $i>0$,}
$$
for all positive $m$ such that $m(\Fscr,H)|m$ and for any relatively nef Cartier divisor $D$ on $X$.
\end{thm}
\begin{proof} Let $g:Y\rightarrow X$ be the $\Q$-Cartierization of $H$ (which exists since $H$ is $f$-ample). Let $\bar{H}:=g^{-1}_*H$ and let $h=f\circ g:Y\rightarrow U$; by theorem \ref{thm:ample<=>qample}, $\bar{H}$ is $h$-ample, and $g$-ample. By lemma \ref{lm:themostusefullemmaintheworld}, for all $m\geq 0$,
$$
g_*\regsh_Y(m\bar{H})=\regsh_X(mH).
$$

Let $D$ be any $f$-nef Cartier divisor; let us denote $\Fscr(D)=\Fscr\otimes\regsh_X(D)$. By the projection formula,
\begin{eqnarray}\label{eq:projection}
g_*\big(g^*\Fscr(D)\otimes\regsh_Y(m\bar{H})\big)&\cong&\Fscr\otimes\regsh_X(D)\otimes g_*\regsh_Y(m\bar{H})\cong\\
&\cong&\Fscr\otimes\regsh_X(D+mH)\nonumber
\end{eqnarray}
for all $m\geq0$. The sheaf $g^*\regsh_X(D)$ is $g$-trivial while $g^*\Fscr$ is still coherent and $\regsh_Y(\bar{H})$ is $g$-ample; thus there exists $m_1:=m_1(H)$ such that, for all $m_1|m$,
\begin{eqnarray}\label{eq:R^i=0}
R^jg_*\big(g^*\Fscr(D)\otimes\regsh_Y(m\bar{H})\big)=0,\quad\textrm{for all $j>0$}
\end{eqnarray}
(this is the relative version of Serre's vanishing, \cite[1.5]{keeler}). We have Grothendieck's spectral sequence
$$
R^if_*\Big(R^jg_*\big(g^*\Fscr(D)\otimes\regsh_Y(m\bar{H})\big)\Big)\Rightarrow R^{i+j}h_*\big(g^*\Fscr(D)\otimes\regsh_Y(m\bar{H})\big).
$$
By \eqref{eq:R^i=0}, the Grothendieck's spectral sequence converges at the $E_2$ page, that is, $E_2=E_\infty$, and we have the equalities
\begin{eqnarray*}
R^if_*\Big(g_*\big(g^*\Fscr(D)\otimes\regsh_Y(m\bar{H})\big)\Big)=R^ih_*\big(g^*\Fscr(D)\otimes\regsh_Y(m\bar{H})\big)
\end{eqnarray*}
for every $m_1|m$ and every $i\geq0$. By \eqref{eq:projection}, the cohomology on the left-hand of the equality is
$$
R^if_*\Big(g_*\big(g^*\Fscr(D)\otimes\regsh_Y(m\bar{H})\big)\Big)=R^if_*(\Fscr\otimes\regsh_X(D+mH)),
$$
for all $m_0|m$. On the other hand, since $g^*D$ is $h$-nef on $Y$ (it is the pullback of an $f$-nef Cartier divisor), $g^*\Fscr$ is coherent and $\regsh_Y(\bar{H})$ is $h$-ample, by the relative Fujita's vanishing theorem, \cite[1.5]{keeler}, there exists $m_2(H,g^*\Fscr)=m_2(H,\Fscr)$ such that, for all $m_2:=m_2(H,\Fscr)|m$ and for all $i>0$,
$$
R^ih_*\big(g^*\Fscr(D)\otimes\regsh_Y(m\bar{H})\big)=R^ih_*\big(g^*\Fscr\otimes\regsh_Y(m\bar{H}+g^*D)\big)=0.
$$
Thus, if $m(H,\Fscr)=lcm(m_1,m_2)$, for all $i>0$ and positive $m(H\Fscr)|m$,
$$
R^if_*(\Fscr\otimes\regsh_X(mH+D))=0.
$$
\end{proof}
\begin{cor}[Fujita vanishing for locally free sheaves] Let $X$ be a projective variety, let $H$ be an ample Weil divisor and let $\Fscr$ be any locally free coherent sheaf on $X$. There exists an integer $m(\Fscr,H)$ such that
$$
H^i(X,\Fscr\otimes\regsh_X(mH+D))=0,\quad\textrm{for all $i>0$,}
$$
for all positive $m$ such that $m(\Fscr,H)|m$ and for any nef Cartier divisor $D$ on $X$.
\end{cor}
\begin{rk} It is not clear if the hypotheses can be weakened to allow any coherent sheaf for $\Fscr$ or nef Weil divisors.
\end{rk}
\subsection{Kawamata-Viehweg vanishing}
With a similar technique to the one used in the proof of the relative Fujita vanishing, we can prove the relative Kawamata-Viehweg.
\begin{thm}[Relative Kawamata-Viehweg vanishing] Let $f:X\rightarrow U$ be a projective morphism of complex quasi-projective varieties. Let $(X,\Delta)$ be a klt pair and let $D$ be a Weil divisor on $X$ such that $D-(K_X+\Delta)$ is $f$-big and $f$-nef. Then
$$
R^if_*\regsh_X(D)=0,\quad\textrm{for $i>0$}.
$$
\end{thm}
\begin{proof} Since $(X,\Delta)$ is klt, both $\Rscr(X,D)$ and $\Rscr(X,D-K_X-\Delta)$ are finitely generated, \cite[92]{kol}. Let $g:Y=\Proj_X\Rscr(X,D)\rightarrow X$, let $\bar{D}=g^{-1}_*D$ and let $h=f\circ g:Y\rightarrow U$. Let $\Delta_Y=g^{-1}_*\Delta$. Since $g$ is small $K_Y+\Delta_Y=g^*(K_X+\Delta)$ and $g_*\Delta_Y=\Delta$. By \cite[2.30]{komo}, $(Y,\Delta_Y)$ is a klt pair. Moreover, since $\bar{D}$ is $g$-ample and $K_Y+\Delta_Y=g^*(K_X+\Delta)$ is $g$-trivial, $\bar{D}-(K_Y+\Delta_Y)=\bar{D}-g^*(K_X+\Delta)$ is still $g$-ample and $\bar{D}-(K_Y+\Delta_Y)=g^{-1}_*\big(D-(K_X+\Delta)\big)$. By \ref{thm:nef<=>qnef} and \ref{thm:big<=>qbig}, $\bar{D}-(K_Y+\Delta_Y)$ is $\Q$-Cartier, $h$-nef and $h$-big. By the relative Kawamata-Viehweg vanishing,
$$
R^ih_*\regsh_Y(\bar{D})=0,\quad\textrm{for $i>0$}.
$$
On the other hand, by \ref{lm:themostusefullemmaintheworld}, $g_*\regsh_Y(\bar{D})=\regsh_X(D)$, and by the relative Kawamata-Viehweg vanishing
$$
R^ig_*\regsh_Y(\bar{D})=0,\quad\textrm{for $i>0$},
$$
since $\bar{D}-(K_Y+\Delta_Y)$ is $g$-ample. Thus the Grothendieck spectral sequence
$$
R^if_*R^jg_*\regsh_Y(\bar{D})\Rightarrow R^{i+j}h_*\regsh_Y(\bar{D})
$$
converges immediately, and we have
$$
R^if_*\regsh_X(D)=R^if_*\big(g_*\regsh_Y(\bar{D})\big)\cong R^ih_*\regsh_Y(\bar{D})=0,\quad\textrm{for $i>0$}.
$$
\end{proof}
\begin{cor}[Kawamata-Viehweg vanishing] Let $X$ be a complex projective variety; let $(X,\Delta)$ be a klt pair. Let $D$ be a Weil divisor such that $D-(K_X+\Delta)$ is nef and big. Then
$$
H^i(X,\regsh_X(D))=0,\quad\textrm{for $i>0$}.
$$
\end{cor}
%
%
%
%
%          Non-Q-Gor MMP
%
%
%
%
\section{First steps towards a non-$\Q$-Gorenstein MMP}
In this section we work over an algebraically closed field $k$ of characteristic $0$. The first results that are needed for a non-$\Q$-Gorenstein Minimal Model Program are a non-vanishing theorem, a global generation statement, and an understanding of Fano-type varieties. This is done here, in \S5.1 and \S5.2 respectively. The next step is to prove a cone and a contraction theorem. These theorems and a proposal for a non-$\Q$-Gorenstein MMP are investigated by the first author in \cite{MMPnoflips}.
\subsection{Non-vanishing and Global generation}
\label{nvan}
We collected these two results together as they show peculiar aspects of Weil divisors, which make the original results, in some sense, impossible to fully generalize. Moreover, the technique of the proof is the same.

We will start by proving a non-$\Q$-Cartier version of Shokurov's Non-vanishing theorem, see \cite[3.4]{komo}. Next we will prove a generalization of the Base-point free theorem, see \cite[3.3]{komo}. However, for Weil divisors, base-point freeness and global generation are not equivalent. The non-expert can think of the case of an affine cone and a non-$\Q$-Cartier divisor on it. Since the variety is affine, every coherent sheaf is globally generated; on the other hand, this divisor (and all its powers) must go through the vertex of the cone. So, our generalized version of the Base-point freeness will be only a statement about global generation.
\begin{prop}[Non-vanishing] Let $f:X\rightarrow U$ be a projective morphism of quasi-projective varieties over $k$. Let $(X,\Delta)$ be a klt pair and let $D$ be an $f$-nef divisor such that $aD-(K_X+\Delta)$ is $f$-nef and $f$-big for some $a>0$. Then, for all positive $1||m$, $mD+\lceil-\Delta\rceil\sim_{\Q,f}F\geq0$.
\end{prop}
\begin{proof} Without loss of generality, we can assume that $a=1$. Since $(X,\Delta)$ is klt, the algebras $\Rscr(X,D)$ and $\Rscr(X,D-K_X-\Delta)$ are both finitely generated. Let $g:Y=\Proj_X\Rscr(X,D)\rightarrow X$, let $\bar{D}=g^{-1}_*D$ and let $h=f\circ g:Y\rightarrow U$. Let $\Delta_Y=g^{-1}_*\Delta$. Since $g$ is small $K_Y+\Delta_Y=g^*(K_X+\Delta)$ and $g_*\Delta_Y=\Delta$. By \cite[2.30]{komo}, $(Y,\Delta_Y)$ is a klt pair. Moreover, since $\bar{D}$ is $g$-ample and $K_Y+\Delta_Y=g^*(K_X+\Delta)$ is $g$-trivial, $\bar{D}-(K_Y+\Delta_Y)=\bar{D}-g^*(K_X+\Delta)$ is still $g$-ample and $\bar{D}-(K_Y+\Delta_Y)=g^{-1}_*\big(D-(K_X+\Delta)\big)$. By \ref{thm:nef<=>qnef} and \ref{thm:big<=>qbig}, $\bar{D}-(K_Y+\Delta_Y)$ is $\Q$-Cartier, $h$-nef and $h$-big and $\bar{D}$ is $\Q$-Cartier $h$-nef. So we can apply the usual non-vanishing theorem and deduce that $m\bar{D}+\lceil-\Delta_Y\rceil\sim_{\Q,h}F_Y\geq0$.

Let $G=\lceil-\Delta\rceil$ and $G_Y=\lceil-\Delta_Y\rceil=g_*^{-1}G$. Since $g_*\regsh_Y(m\bar{D})=\regsh_X(mD)$ and $G$ and $G_Y$ are effective, $g_*\regsh_Y(m\bar{D}+G_Y)\subseteq\regsh_X(mD+G)$, which in turn implies that $h_*\regsh_Y(m\bar{D}+G_Y)\subseteq f_*\regsh_X(mD+G)$. Hence $mD+G=mD+\lceil-\Delta\rceil\sim_{f,\Q} F\geq0$.
\end{proof}
\begin{prop}[Global generation] Let $(X,\Delta)$ be a klt pair, $f:X\rightarrow U$ be a projective morphism of quasi-projective normal varieties. Let $D$ be an $f$-nef divisor such that $aD-(K_X+\Delta)$ is $f$-nef and $f$-big for some $a>0$. Then $D$ is $f$-agg (relatively asymptotically globally generated).
\end{prop}
\begin{proof} Since $(X,\Delta)$ is klt, $\Rscr(X,D)$ is finitely generated. Let $g:Y\rightarrow X$ be the $\Q$-Cartierization of $D$, $\bar{D}=g_*^{-1}D$, and $h=f\circ g$. As in the previous proof, if $\Delta_Y=g^{-1}_*\Delta$, $(Y,\Delta_Y)$ is klt, $\bar{D}$ is $g$-ample and $h$-nef, while $a\bar{D}-(K_Y+\Delta_Y)$ is $h$-nef and $h$-big. By the base-point free theorem, $m\bar{D}$ is $h$-point free for positive $1||m$, that is, $\bar{D}$ is $h$-agg. Since $g_*\regsh_Y(m\bar{D})=\regsh_X(mD)$, $D$ is $f$-agg.
\end{proof}
%
%
%
%
%          Fano
%
%
%
%
\subsection{Fano varieties}
\label{fano} 
\begin{df} Let $X$ be a projective variety. The variety $X$ is called \emph{log Fano} if there exists an effective boundary $\Delta$ such that $K_X+\Delta$ is $\Q$-Cartier and anti-ample, and $(X,\Delta)$ has klt singularities.
\end{df}
\begin{rk} Some authors add the extra assumption of $\Q$-factoriality in the above definition.
\end{rk}

\begin{thm} Let $X$ be a projective variety having klt singularities (in the sense that there exists a boundary $\Gamma$ such that $(X,\Gamma)$ is a klt pair) and such that $-K_X$ is ample (not necessarily $\Q$-Cartier). Then $X$ is log Fano.
\end{thm}
\begin{proof} Since $\Rscr(X,-K_X)$ is finitely generated there exists a small map $g:Y=\Proj_X\Rscr(X,-K_X)\rightarrow X$, such that $g^{-1}_*(-K_X)$ is $g$-ample and ample.Notice that, the map $g$ being small, $-K_Y=g_*^{-1}(-K_X)$ and $K_Y=g^\natural(K_X)$. As in \cite[5.14]{ltsings}, we can show that $Y$ has log terminal singularities. Indeed, since $X$ has klt singularities, there exists an $m>0$ such that, for any resolution $f:Z\rightarrow X$, $K_Z-\frac{1}{m}f^\natural(mK_X)>-1$. Let $Z$ be any such log resolution, which without loss of generality we can assume factors through $g$ as $f=g\circ h$, for some $h:Z\rightarrow Y$. Since $g^{-1}_*(-K_X)=-K_Y$ is $g$-ample, $\regsh_Y(-mK_Y)$ is globally generated over $X$ for $m$ sufficiently divisible, hence we have an isomorphism at the level of sheaves, that is, $\regsh_X(-K_X)\cdot\regsh_Y\cong\regsh_Y(-K_Y)$. Thus
$$
K_Z-h^*(K_Y)=K_Z-\frac{1}{m}h^*(mK_Y)=K_Y-\frac{1}{m}f^\natural(mK_X)>-1,
$$
where the last equality holds by [dFH09, Lemma 2.7]. Thus $Y$ (which is $\Q$-Gorenstein) has log terminal singularities.

Let now $A$ be any ample Cartier divisor on $X$. Since $-K_Y$ is ample, there exists $b>1$ such that $-bK_Y-g^*A$ is globally generated. Let $M$ be the general element in the system $M\in|\regsh_X(-bK_Y-g^*A)|$. Notice that $-bK_Y-bM\sim g^*A$, which is trivial on the fibers. Let $\Delta:=g_*M/b$. Since $K_Y+M/b\sim_\Q-g^*A$, $K_X+\Delta$ is $\Q$-Cartier, and $K_X+\Delta\sim_\Q-A$. Finally, since $Y$ has log terminal singularities, and $M$ was general, $(Y,M/b)$ has klt singularities. Notice that $g_*(M/b)=\Delta$ and $K_Y+M/b=g^*(K_X+\Delta)$ (by construction). Thus $(X,\Delta)$ has klt singularities \cite[2.30]{komo}.
\end{proof}

\begin{rk} In \cite[5.2]{ltsings}, the notion of \emph{log terminal${}^+$} singularities was introduced: if $X$ is a normal variety over $\C$, $X$ has \emph{lt${}^+$} singularities if for one (equivalently all) log resolution $f:Y\rightarrow X$, $K_Y+f^*(-K_X)$ has coefficients strictly bigger than $-1$ (for any prime component of the exceptional divisor). By \cite[5.15]{ltsings}, if $X$ has lt${}^+$ singularities, than is has klt singularities (in the sense of the above theorem) if and only if $\Rscr(X,-K_X)$ is finitely generated. Thus the above theorem is true under the hypothesis that $X$ has only lt${}^+$ singularities. Since lt${}^+$ is a broader class than klt, this suggests that the positivity condition on $-K_X$ affects the type of singularity. See \cite{karl-karen} for a phenomenon opposite in nature.
\end{rk}

The above statement is in particular satisfied by normal toric varieties. We will give the following slightly stronger example to show how we aim to use the Theorem in view of the Minimal Model Program.

\begin{ex}
Let $\pi: Z \to X$ be a small morphism of toric varieties such that $K_Z$ is ample and $\Q$-Cartier over $X$. Then for any $K_Z$-negative curve $C$, there exists a boundary $\Delta$ on $X$ such that $(K_X +\Delta).\pi_*C <0$. Indeed, let $C \subset Z$ such that $K_Z.C<0$. Since $K_Z$ is $\pi$-ample, $C$ is not contracted to a curve. Let $D$ be a very ample effective Cartier divisor on $X$, of the form $D= \sum d_{\rho}D_{\rho}$; in particular $D.\pi_*C >0$. We can assume that there exists an integer $k>0$ such that $D' :=D/k$ is a $\Q$-Cartier divisor such that $|d_{\rho}/k| <1$ for every $\rho \in \Sigma(1)$. Since $X$ is toric, $K_X=\sum -D_\rho$. Let us choose $\Delta = -K_X -D'$, that by assumption is an effective divisor and $K_X+ \Delta = -D'$ is $\Q$-Cartier and $(K_X+ \Delta).\pi_*C <0$. Again, notice that, since $X$ is toric and given the generality of the construction of $\Delta$, $(X,\Delta)$ is a klt pair. 

The next step is to show that some contractions on $Z$ induce contractions on $X$. Let $C$ be a curve in $X$ which generates an extremal ray of the cone of curves. Since the cone of curves of $X$ is the projection of the cone of curves on $Z$, there is a curve $C'$ on $Z$ which generates an extremal ray on $Z$ which maps onto the ray generated by $C$. By the discussion in the above paragraph, if $C'$ is $K_Z$-negative, there is a boundary $\Delta$ on $X$ such that $(X,\Delta)$ is a klt pair and $(K_X+\Delta).C<0$. Thus $C$ generates a $(K_X+\Delta)$-negative extremal ray. By the cone theorem, both $C$ and $C'$ induce contractions. We obtain the following diagram

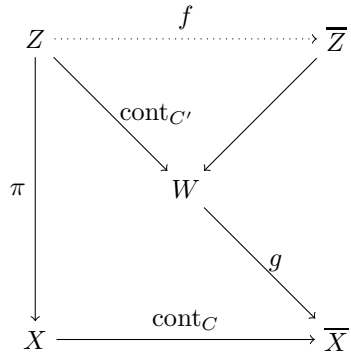
\begin{center}
\begin{tikzpicture}
\node (A) at (-2,2) {$Z$};
\node (B) at (-2,-2) {$X$};
\node (C) at (2,-2){$\overline{X}$};
\node (D) at (2,2){$\overline{Z}$};
\node (E) at (0,0) {$W$};
\path[->] (A) edge[dotted] node[above]{$f$}(D);
\path[->] (A) edge node[left]{$\pi$}(B);
\path[->] (B) edge node[above]{cont${}_C$}(C);
\path[->] (A) edge node[right]{cont${}_{C'}$}(E);
\path[->] (E) edge node[right]{$g$}(C);
\path[->] (D) edge (E);
\end{tikzpicture}
\end{center}
(notice that the map $g$ must exists since $K_Z$ is $\pi$-ample and ${\rm cont}_{C'}$-antiample).
\end{ex}

%----------BIBLIOGRAPHY---------------------

    \bigskip
   Alberto Chiecchio,
   Department of Mathematics, 
   University of Arkansas, 
   SCEN 353, Fayetteville, AR 72701,
   E-mail: chieccho@uark.edu

   \bigskip
   Stefano Urbinati,
   Universit\`a degli Studi di Padova,
   Dipartimento di Matematica,
   ROOM 608, Via Trieste 63,
   35121 Padova (Italy),
   Email: urbinati.st@gmail.com

\end{document}